\documentclass[a4paper,11pt]{amsart}

\usepackage{graphics}

\usepackage[T1]{fontenc}    

\usepackage{amsthm}
\usepackage{amsbsy,amsmath,amssymb,amscd,amsfonts}
\usepackage[pagebackref=true]{hyperref}
\usepackage{url}            
\usepackage{booktabs}       
\usepackage{nicefrac}       
\usepackage{microtype}      

\usepackage{graphicx,float,latexsym,color}
\usepackage[font={scriptsize,it}]{caption}
\usepackage{subcaption}

\usepackage{makecell}

\usepackage[dvipsnames]{xcolor}

\newtheorem{theorem}{Theorem}

\newtheorem{proposition}{Proposition}

\newtheorem{remark}{Remark}

\hypersetup{
    pdftoolbar=true,        
    pdfmenubar=true,        
    pdffitwindow=false,     
    pdfstartview={FitH},    
    colorlinks=true,       
    linkcolor=OliveGreen,          
    citecolor=blue,        
    filecolor=black,      
    urlcolor=red           
}

\usepackage{lineno}

\title[]{The Circumbilliard: Any Triangle\\can be a 3-Periodic}

\author{Dan Reznik}
\address{Dan Reznik,
Data Science Consulting,\\
Rio de Janeiro, RJ, Brazil}
\email{dan@dat-sci.com}

\author{Ronaldo Garcia}
\address{Ronaldo Garcia,
Inst. de Matemática e Estatística,\\
Univ. Federal de Goiás,\\
Goiânia, GO, Brazil}
\email{ragarcia@ufg.br}

\begin{document}
\maketitle
\begin{abstract}
A Circumconic passes through a triangle's vertices. We define the Circumbilliard, a circumellipse to a generic triangle for which the latter is a 3-periodic. We study its properties and associated loci.

\end{abstract}

\section{Introduction}
\label{sec:intro}
Given a triangle, a {\em circumconic} passes through its three vertices and satifies two additional constraints, e.g., center or perspector\footnote{Where reference and polar triangles are perspective \cite{mw}.}. We study properties and invariants of such conics derived from a 1d family of triangles: 3-periodics in an Elliptic Billiard (EB): these are triangles whose bisectors coincide with normals to the boundary (bounces are elastic), see Figure~\ref{fig:three-orbits-proof}.

\begin{figure}
    \centering
    \includegraphics[width=.66\textwidth]{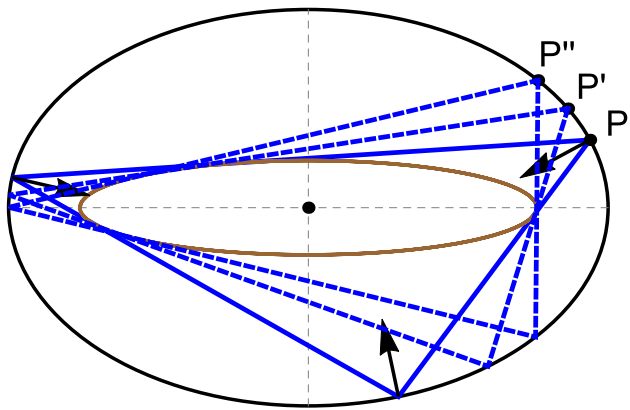}
    \caption{3-periodics (blue) in the Elliptic Billiard (EB, black): normals to the the boundary at vertices (black arrows) are bisectors. The family is constant-perimeter and envelopes a confocal Caustic (brown). This family conserves the ratio inradius-to-circumradius and has a stationary Mittenpunkt at the EB center. \textbf{Video}: \cite[PL\#01]{reznik2020-playlist-circum}.}
    \label{fig:three-orbits-proof}
\end{figure}

Amongst all planar curves, the EB is uniquely integrable \cite{kaloshin2018}. It can be regarded as a special case of Poncelet's Porism \cite{dragovic11}. These two propeties imply two classic invariances: $N$-periodics have constant perimeter and envelop a confocal Caustic. The seminal work is \cite{sergei91} and more recent treatments include \cite{lynch2019-billiards,rozikov2018}. 

We have shown 3-periodics also conserve the Inradius-to-Circumradius\footnote{As does the Poristic Family \cite{gallatly1914-geometry}.} ratio which implies an invariant sum of cosines, and that their {\em Mittenpunkt}\footnote{Where lines drawn from each Excenter thru sides' midpoints meet.} is stationary at the EB center \cite{reznik2020-intelligencer}. Indeed many such invariants have been effectively generalized for $N>3$ \cite{akopyan2020-invariants,bialy2020-invariants}.

We have also studied the loci of 3-periodic Triangle Centers over the family: out of the first 100 listed in \cite{etc}, 29 sweep out ellipses (a remarkable fact on its own) with the remainder sweeping out higher-order curves \cite{garcia2020-ellipses}. Related is the study of  loci described by the Triangle Centers of the Poristic Triangle family \cite{odehnal2011-poristic}.

\textbf{Summary of the paper}: given a generic triangle $T$ we define its {\em Circumbilliard} CB: a Circumellipse to $T$ for which the latter is a 3-periodic. We then analyze the dynamic of geometry of Circumbilliards for triangles derived from the 3-periodic family such as the Excentral, Anticomplementary, Medial, and Orthic, as well as the loci swept by their centers. Additional results include:

\begin{itemize}
\item Proposition~\ref{prop:right-triangle} in Section~\ref{sec:cb_derived} describes regions of the EB which produce acute, right-triangle, and obtuse 3-periodics. 
\item Theorem~\ref{thm:poristic} in Section~\ref{sec:cb_derived}: The aspect ratio of Circumbilliards of the Poristic Triangle Family \cite{gallatly1914-geometry} is invariant. This is a family of triangle with fixed Incircle and Circumcircle.
\end{itemize}

A reference table with all Triangle Centers, Lines, and Symbols appears in Appendix~\ref{app:symbols}. Videos of many of the experiments are assembled on Table~\ref{tab:playlist} in Section~\ref{sec:conclusion}.


\section{The Circumbilliard}
\label{sec:cb}
Let the boundary of the EB satisfy:

\begin{equation}
\label{eqn:billiard-f}
f(x,y)=\left(\frac{x}{a}\right)^2+\left(\frac{y}{b}\right)^2=1.
\end{equation}

Where $a>b>0$ denote the EB semi-axes, and $c=a^2-b^2$ throughout the paper. Below we use {\em aspect ratio} as the ratio of an ellipse's semi-axes. When referring to Triangle Centers we adopt Kimberling $X_i$ notation \cite{etc}, e.g., $X_1$ for the Incenter, $X_2$ for the Barycenter, etc., see Table~\ref{tab:kimberling} in Appendix~\ref{app:symbols}.

The following five-parameter equation is assumed for all circumconics not passing through $(0,0)$.

\begin{equation}
1 + c_1 x + c_2 y + c_3 x y + c_4 x^2 + c_5 y^2=0
\label{eqn:e0}
\end{equation}

\begin{proposition}
Any triangle $T=(P_1,P_2,P_3)$ is associated with a unique 
ellipse $E_9$ for which $T$ is a billiard 3-periodic. The center of $E_9$ is T's Mittenpunkt.
\end{proposition}

\begin{proof}
If $T$ is a 3-periodic of $E_9$, by Poncelet's Porism, $T$ is but an element of a 1d family of 3-periodics, all sharing the same confocal Caustic\footnote{This turns out to be the Mandart Inellipse $I_9$ of the family \cite{mw}.}. This family will share a common Mittenpunkt $X_9$ located at the center of $E_9$  \cite{reznik2020-intelligencer}. Appendix~\ref{app:circum-linear} shows how to obtain the parameters for \eqref{eqn:e0} such that it passes through $P_1,P_2,P_3$ and is centered on $X_9$: this yields a $5{\times}5$ linear system. Solving it its is obtained a quadratic equation with positive discriminant, hence the conic is an ellipse.
\end{proof}

$E_9$ is called the Circumbilliard (CB) of $T$. Figure~\ref{fig:circumbilliard} shows examples of CBs for two sample triangles.

\begin{figure}[H]
    \centering
    \includegraphics[width=.8\textwidth]{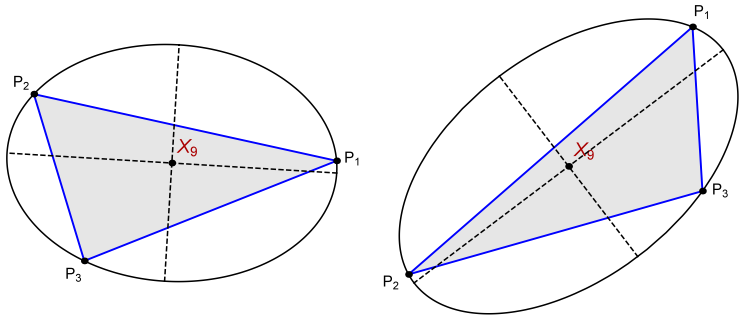}
    \caption{Two random triangles are shown as well as their Circumbilliards (CBs). Notice their axes in general are not horizontal/vertical. An algorithm for computing the CB is given in Appendix~\ref{app:circum-linear}. \textbf{Video:} \cite[PL\#02]{reznik2020-playlist-circum}}
    \label{fig:circumbilliard}
\end{figure}

\section{Circumbilliards of Derived Triangles}
\label{sec:cb_derived}
Figure~\ref{fig:cb_trio} shows CBs for the Excentral, Anticomplementary (ACT), and Medial Triangles, derived from 3-periodics.

\begin{figure}
    \centering
    \includegraphics[width=\textwidth]{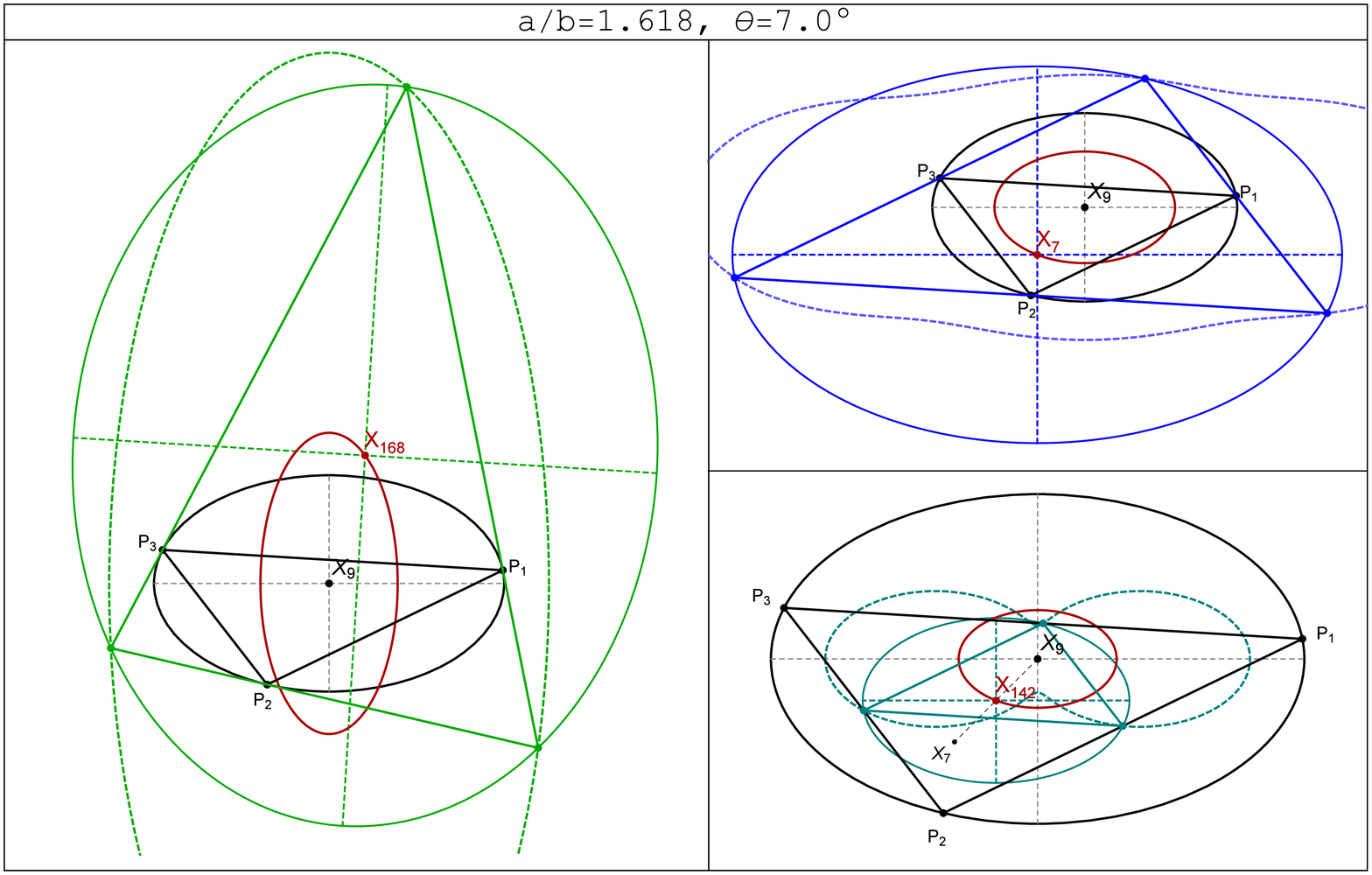}
    \caption{Draw in black in each picture is an $a/b=\varphi{\simeq}1.618$ EB and a 3-periodic at $t=7.0^\circ$. \textbf{Left}: the CB of the Excentral Triangle (solid green) centered on the latter's Mittenpunkt is $X_{168}$ \cite{etc}. Its locus (red) is non-elliptic. Also shown (dashed green) is the elliptic locus of the Excenters (the MacBeath Circumellipse $E_6'$ of the Excentrals \cite{mw}), whose center is $X_9$ \cite{garcia2020-ellipses}; \textbf{Top Right}: the CB of the Anticomplementary Triangle (ACT) (blue), axis-aligned with the EB. Its center is the Gergonne Point $X_7$, whose locus (red) is elliptic and similar to the EB \cite{garcia2020-ellipses}. The locus of the ACT vertices is not elliptic (dashed blue); \textbf{Bottom Right}: the CB of the Medial Triangle (teal), also axis-aligned with the EB, is centered on $X_{142}$, whose locus (red) is also elliptic and similar to the EB, since it is the midpoint of $X_9X_7$ \cite{etc}. The locus of the medial vertices is a dumb-bell shaped curve (dashed teal). \textbf{Video:} \cite[PL\#03]{reznik2020-playlist-circum}}
    \label{fig:cb_trio}
\end{figure}

\subsection{Excentral Triangle}
\label{sec:cb_exc}
The locus of the Excenters is shown in Figure~\ref{fig:cb_trio} (left). It is an ellipse similar to the $90^\circ$-rotated locus of $X_1$ and its axes $a_e,b_e$ are given by \cite{garcia2019-incenter,garcia2020-ellipses}:

\begin{equation*}
 a_e=\frac{{b}^{2}+\delta}{a},\;\;\; 
 b_e=\frac{{a}^{2}+\delta}{b}
\end{equation*}

Where $\delta=\sqrt{a^4-a^2 b^2+b^4}$. 

\begin{proposition}
The locus of the Excenters the stationary MacBeath Circumellipse $E_6'$ \cite{mw} of the Excentral Triangles.
\end{proposition}

\begin{proof}
The center of $E_6'$ is the Symmedian Point $X_6$ \cite[MacBeath Circumconic]{mw}. The Excentral Triangle's $X_6$ coincides with the Mittenpunkt $X_9$ of the reference \cite{etc}. Since over the 3-periodics the vertices of the Excentral lie on an ellipse and its center is stationary, the result follows.
\end{proof}

\begin{proposition}
The Excentral CB is centered on $X_{168}$, whose trilinears are irrational, and whose locus is non-elliptic.
\end{proposition}

\begin{proof}
$X_{168}$ is the Mittenpunkt of the Excentral Triangle \cite{etc} and its trilinears are irrational\footnote{No Triangle Center whose trilinears are irrational on sidelengths has yet been found whose locus under the 3-periodic family is an ellipse \cite{garcia2020-ellipses}.} on the sidelengths. To determine if its locus is an ellipse we use the algebro-numeric techniques described in \cite{garcia2020-ellipses}. Namely, a least-squares fit of a zero-centered, axis-aligned ellipse to a sample of $X_{168}$ positions of the 3-periodic family produces finite error, therefore it cannot be an ellipse.
\end{proof}

This had been observed in \cite{garcia2020-ellipses} for several irrational centers such as $X_i,i=$13--18, as well as many others. Notice a center may be rational but produce a non-elliptic locus, the emblematic case being $X_6$, whose locus is a convex quartic. Other examples include $X_j,j=$19, 22--27, etc.

\subsection{Anticomplementary Triangle (ACT)}
\label{sec:cb_act}
The ACT is shown in Figure~\ref{fig:cb_trio} (top right). The locus of its vertices is clearly not an ellipse.

The ACT is perspective with the reference triangle (3-periodic) at $X_2$ and all of its triangle centers correspond to the anticomplement\footnote{Anticomplement: a 1:2 reflection about $X_2$.} of corresponding reference ones \cite{mw}. The center of the CB of the ACT is therefore $X_7$, the anticomplement of $X_9$. We have shown the locus of $X_7$ to be an ellipse similar to the EB with axes \cite{garcia2020-ellipses}:

\begin{equation*}
\left(a_7,b_7\right)=k\left(a,b\right),\;\;\text{with: }k=\frac{2\delta - {a}^{2}-{b}^{2}}{c^2}
\end{equation*}

\begin{remark}
The axes of the ACT CB are parallel to the EB and of fixed length.
\end{remark}

This stems from the fact the ACT is homothetic to the 3-periodic.

\subsection{Medial Triangle}
\label{sec:cb_medial}
The locus of its vertices is the dumbbell-shaped curve, which at larger $a/b$ is self-intersecting, and therefore clearly not an ellipse, Figure~\ref{fig:cb_trio} (bottom right).

Like the ACT, the Medial is perspective with the reference triangle (3-periodic) at $X_2$. All of its triangle centers correspond to the complement\footnote{Complement: a 2:1 reflection about $X_2$.} of corresponding reference ones \cite{mw}. The center of the CB of the Medial is therefore $X_{142}$, the complement of $X_9$. This point is known to sit midway between $X_9$ and $X_7$.

\begin{remark}
The locus of $X_{142}$ is an ellipse similar to the EB.
\end{remark}

This stems from the fact $X_9$ is stationary and the locus of $X_7$ is an ellipse similar to the EB (above). Therefore its axes will be given by:

\begin{equation*}
\left(a_{142},b_{142}\right)=(a_7,b_7)/2
\end{equation*}

Stemming from homothety of 3-periodic and its Medial:

\begin{remark}
The axes of the Medial CB are parallel to the EB and of fixed length.
\end{remark}

\subsection{Superposition of ACT and Medial}
\label{sec:cb_super_act_medial}
\begin{proposition}
The Intouchpoints of the ACT (resp. 3-periodic) are on the EB (resp. on the CB of the Medial)
\end{proposition}

\begin{proof}
The first part was proved in \cite[Thm. 2]{reznik2020-ballet}. Because the 3-periodic can be regarded as the ACT of the Medial, the result follows.
\end{proof}

This phenomenon is shown in Figure~\ref{fig:cb_act_med_superposed}. Also shown is the fact that $X_i,i=$7, 142, 2, 9, 144 are all collinear and their intermediate intervals are related as $3:1:2:6$. In \cite{etc_central_lines} this line is known as $L(X_2,X_7)$ or $\mathcal{L}_{663}$. $X_{144}$ is the perspector of the ACT and its Intouch Triangle (not shown).

\begin{figure}
    \centering
    \includegraphics[width=\textwidth]{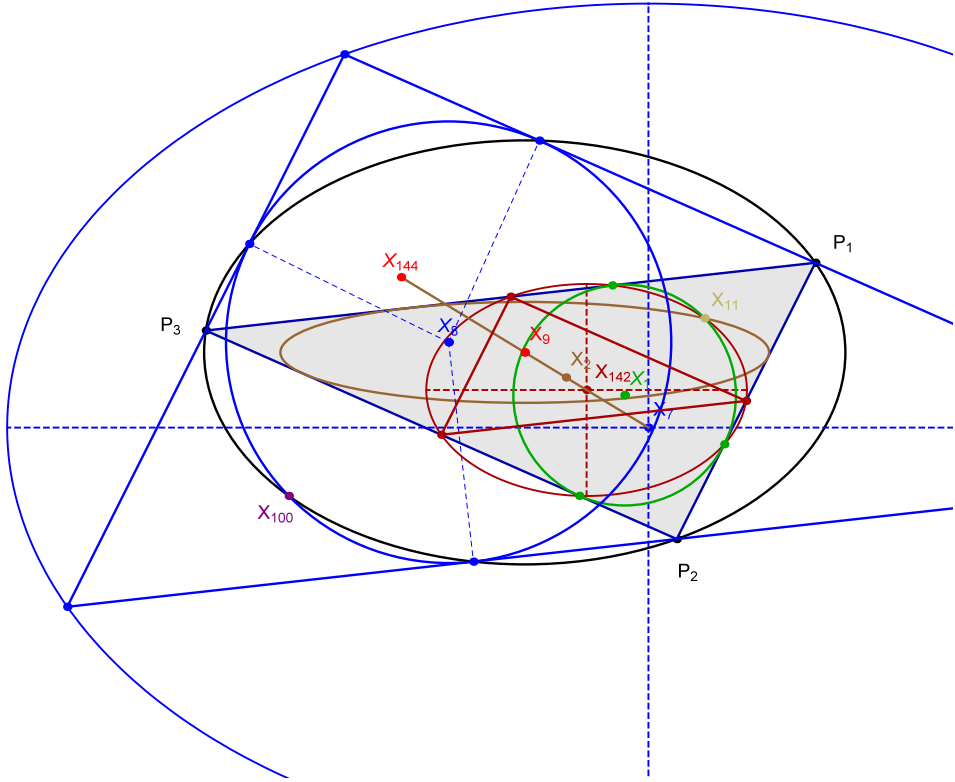}
    \caption{Construction for both ACT and Medial CBs, centered on $X_7$ and $X_{142}$, respectively. The Incircle of the ACT (resp. 3-periodic) is shown blue (resp. green). The former touches the ACT at the EB and the latter touches the 3-periodic sides at the Medial CB. Also shown is line $L(2,7)=L_{663}$ which cointains $X_i$, $i=7,142,2,9,144$. Their consecutive distances are proportional to $3:1:2:6$. $X_{144}$ was included since it is the perspector of the ACT and its Intouch Triangle (not shown) \cite{mw}. \textbf{Video}: \cite[PL\#04,05]{reznik2020-playlist-circum}}
    \label{fig:cb_act_med_superposed}
\end{figure}

\subsection{Orthic Triangle}
\label{sec:cb_ort}
Let $\alpha_4=\sqrt{2\,\sqrt {2}-1}\;{\simeq}\;1.352$. In \cite[Thm. 1]{reznik2020-ballet} we show that if $a/b>\alpha_4$, the 3-periodic family will contain obtuse triangles.

\begin{proposition}
If $a/b>\alpha_4$, the 3-periodic is a right triangle when one of its vertices is at four symmetric points $P^\perp_i$, $i=1,2,3,4$ given by $({\pm}x^\perp,{\pm}y^\perp)$ with:

\begin{equation}
x^\perp=\frac{a^2 \sqrt{a^4+3 b^4-4 b^2 \delta }}{c^3},\;\;\;
y^\perp=\frac{b^2 \sqrt{-b^4-3 a^4+4 a^2 \delta }}{c^3}
\label{eqn:perp}
\end{equation}
\label{prop:right-triangle}
\end{proposition}

\begin{proof}
Let the coordinates of the 3-periodic vertices be $P_1=(x_1,y_1),P_2=(x_2,y_2),P_3=(x_3,y_3)$   as derived in \cite{garcia2019-incenter}.

Computing the equation $\langle P_2-P_1,P_3-P_1\rangle=0$, after careful algebraic manipulations, it follows that $x_1$ satisfies the quartic equation
\[ c^8 x_1^4-2a^4c^2(a^4+3b^4)x_1^2+a^8(a^4+2a^2b^2-7b^4)=0.\] 
For $a/b>\sqrt{2\sqrt{2}-1} $ the only
  positive root in the interval $(0,a)$ is given by

\begin{equation}
x^{\perp}=\frac{a^2 \sqrt{a^4+3 b^4-4 b^2 \delta }}{c^3}.
\label{eqn:xperp}
\end{equation}

\noindent With $y^\perp$ obtainable from \eqref{eqn:billiard-f}.
\end{proof}

Equivalently, a 3-periodic will be obtuse iff one of its vertices lies on top or bottom halves of the EB between the $P_i^\perp$, see Figure~\ref{fig:rect_zones}.

\begin{figure}
    \centering
    \includegraphics[width=.66\textwidth]{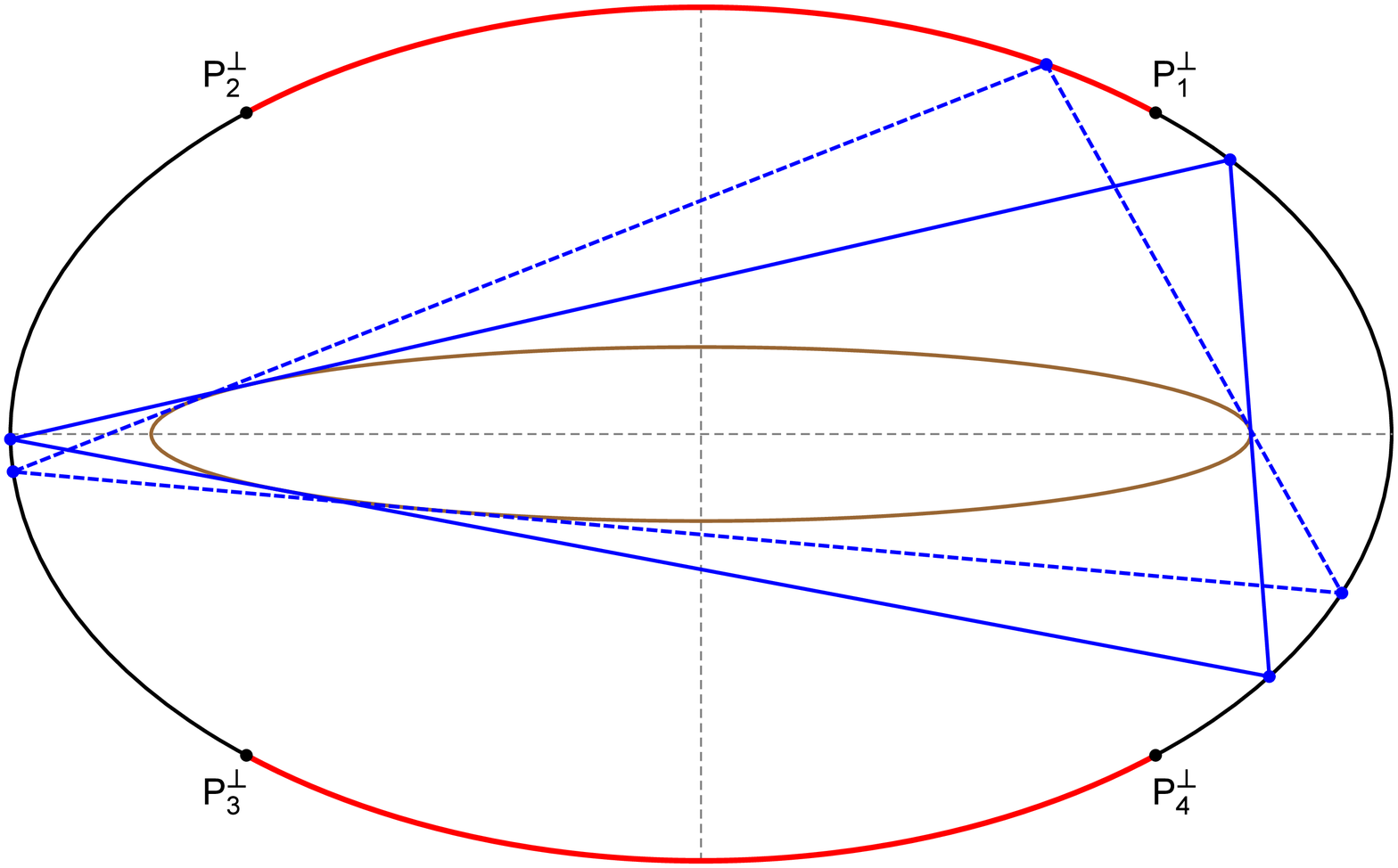}
    \caption{Two 3-periodics are shown: one acute (solid blue) and one obtuse (dashed blue) inscribed into an $a/b=1.618$ EB. Red arcs along the top and bottom halves of the EB indicate that when a 3-periodic vertex is there, the 3-periodic is obtuse. These only exist when $a/b>\alpha_4{\simeq}1.352$.}
    \label{fig:rect_zones}
\end{figure}

Consider the elliptic arc along the EB between $({\pm}x^\perp,y^\perp)$. When a vertex of the 3-periodic lies within  (resp. outside) this interval, the 3-periodic is obtuse (resp. acute).

\begin{proposition}
When $a/b>\alpha_4$, the locus of the center of the Orthic CB has four pieces: 2 for when the 3-periodic is acute (equal to the $X_6$ locus), and 2 when it is obtuse (equal to the locus of $X_6$ of $T''=P_2P_3X_4$.
\end{proposition}

\begin{proof}
It is well-known that \cite{etc} an acute triangle $T$ has an Orthic whose vertices lie on the sidelines. Furthermore the Orthic's Mittenpunkt coincides with the Symmedian $X_6$ of $T$. Also known is the fact that:

\begin{remark}
Let triangle $T'=P_1P_2P_3$ be obtuse on $P_1$. Its Orthic has one vertex on $P_2P_3$ and  two others exterior to $T'$. Its Orthocenter $X_4$ is also exterior. Furthermore, the Orthic's Mittenpunkt is the Symmedian Point $X_6$ of acute triangle $T''=P_2P_3X_4$.
\end{remark}

To see this, notice the Orthic of $T''$ is also\footnote{The anti-orthic pre-images of $T'$ are both the 3-periodic and $T''$.} $T'$. $T''$ must be acute since its Orthocenter is $P_2$.
\end{proof}

The CB of the orthic is shown in Figures~\ref{fig:cb_ort} for four 3-periodic configurations in an EB whose $a/b>\alpha_4$.

\begin{proposition}

The coordinates $({\pm}x^*,{\pm}y^*)$ where the locus of the center of the Orthic's CB transitions from one curve to the other are given by:

\begin{align*}
x^*=&\frac{x^{\perp}}{c^6}\left(a^6+2 a^2 b^4- b^2 \delta(3 a^2  +b^2)    +b^6\right)\\
y^* =&-\frac{y^{\perp}}{c^6}\left(b^6+2 a^4 b^2 -  a^2 \delta (3 b^2 +a^2)   \delta+a^6\right)
\end{align*}

\end{proposition}

\begin{proof}
Let $P_1=(x_1,y_1)$ be the right-triangle vertex of a 3-periodic, given by $(x^\perp,y^\perp)$ as in \eqref{eqn:xperp}. Using \cite{garcia2019-incenter}, obtain $P_2=(p_{2x}/q_2, p_{2y}/q_2)$ and $P_3=(p_{3x}/q_3, p_{3y}/q_3)$, with:
 
 \begin{align*}
     p_{2x}= &b^4 c^2 x_1^3 -2 a^4 b^2 x_1^2 y_1+a^4 c^2 x_1 y_1^2-2 a^6 y_1^3\\
     p_{2y}=&2 b^6 x_1^3-b^4 c^2  x_1^2 y_1+2 a^2 b^4x_1 y_1^2-a^4 c^2 y_1^3\\
     q_2=&b^4 (a^2+b^2) x_1^2-2 a^2 b^2c^2 x_1 y_1 +a^4 (a^2+b^2) y_1^2\\
     p_{3x}=& b^4 c^2 x_1^3  +2 a^4 b^2 x_1^2 y_1+a^4 c^2 x_1 y_1^2+a^6 y_1^3\\
     p_{3y}=&-2 b^6 x_1^3 -b^4 c^2 x_1^2 y_1-2 a^2 b^4 x_1 y_1^2-a^4 c^2 y_1^3 \\
     q_3=&b^4 (a^2+b^2) x_1^2+2 a^2 b^2c^2 x_1 y_1+a^4 (a^2+b^2) y_1^2
 \end{align*}

It can be shown the Symmedian point $X_6$ of a right-triangle is the midpoint of its right-angle vertex altitude. Computing $X_6$ using this property leads to the result.
\end{proof}

Let $\alpha_{eq}=\sqrt{4 \sqrt{3}-3}\,{\simeq}\,1.982$ be the only positive root of $x^4 + 6 x^2 - 39$. It can be shown, see Figure~\ref{fig:cb_ort_equi}:

\begin{proposition}
At $a/b=\alpha_{eq}$, the locus of the Orthic CB is tangent to EB's top and bottom vertices. If a 3-periodic vertex is there, the Orthic is equilateral.
\end{proposition}

\begin{proof}
Let $T$ be an equilateral with side $s_{eq}$ and center $C$. Let $h$ be the distance from any vertex of $T$ to $C$. It can be easily shown that $h/s_{eq}=\sqrt{3}/3$. Let $T'$ be the Excentral Triangle of $T$: its sides are $2s_{eq}$. Now consider the upside down equilateral in Figure~\ref{fig:cb_ort_equi}, which is the Orthic of an upright isosceles 3-periodic. $h$ is clearly the 3-periodic's height and $2s_{eq}$ is its base. The height and width of the upright isosceles are obtained from explicit expressions for the vertices \cite{garcia2019-incenter}:

\begin{align*}
    s_{eq}=\frac{\alpha^2}{\alpha^2-1}\sqrt{2\delta-\alpha^2-1},\;\;\;h=\frac{\alpha^2+\delta+1}{\alpha^2+\delta}
\end{align*}

\noindent where $\alpha=a/b$. Setting $h/s_{eq}=\sqrt{3}/3$ and solving for $\alpha$ yields the required result for $\alpha_{eq}$.
\end{proof}

\begin{figure}
    \centering
    \includegraphics[width=\textwidth]{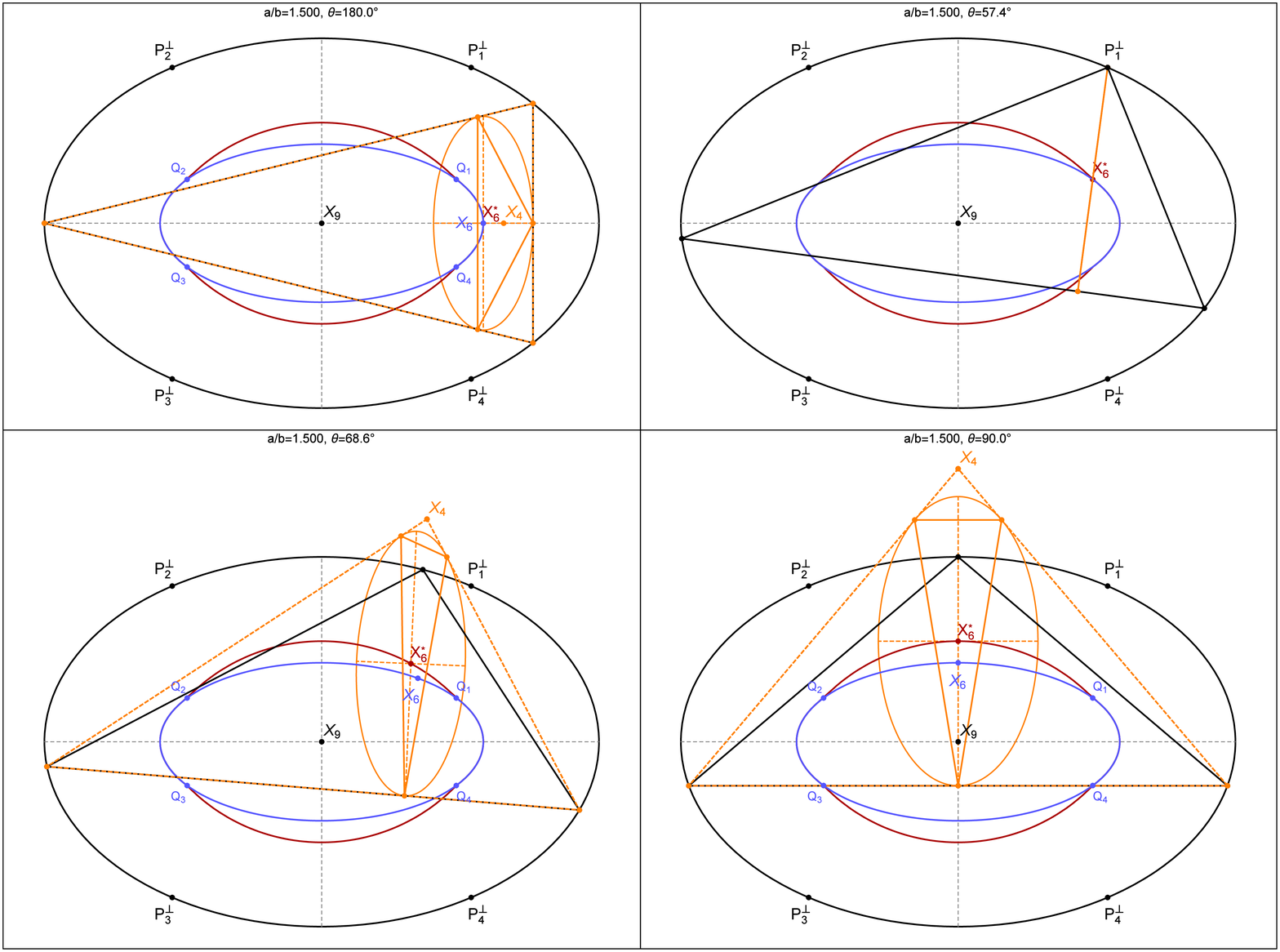}
    \caption{Orthic CB for an EB with $a/b=1.5>\alpha_4$, i.e., containing obtuse 3-periodics, which occur when a 3-periodic vertex lies on the top or bottom areas of the EB between the $P^\perp$. \textbf{Top left}: 3-periodic is sideways isosceles and acute (vertices outside $P^\perp$, so 3 orthic vertices lie on sidelines. The Orthic CB centers is simply the mittenpunkt of the Orthic, i.e, $X_6$ of the 3-periodic (blue curve: a convex quartic \cite{garcia2020-ellipses}). \textbf{Top right}: The position when a vertex is at a $P^\perp$ and the 3-periodic is a right triangle: its Orthic and CB degenerate to a segment. Here the CB center is at a first (of four) transition points shown in the other insets as $Q_i$, $i=1,2,3,4$. \textbf{Bottom left}: The 3-periodic is obtuse, the Orthic has two exterior vertices, and the center of the CB switches to the Symmedian of $T''=P_1P_2X_4$ (red portion of locus). \textbf{Bottom right:}. The 3-periodic is an upright isosceles, still obtuse, the center of the Orthic CB reaches its highest point along its locus (red). \textbf{Video}: \cite[PL\#06]{reznik2020-playlist-circum}.}
    \label{fig:cb_ort}
\end{figure}

\begin{figure}
    \centering
    \includegraphics[width=\textwidth]{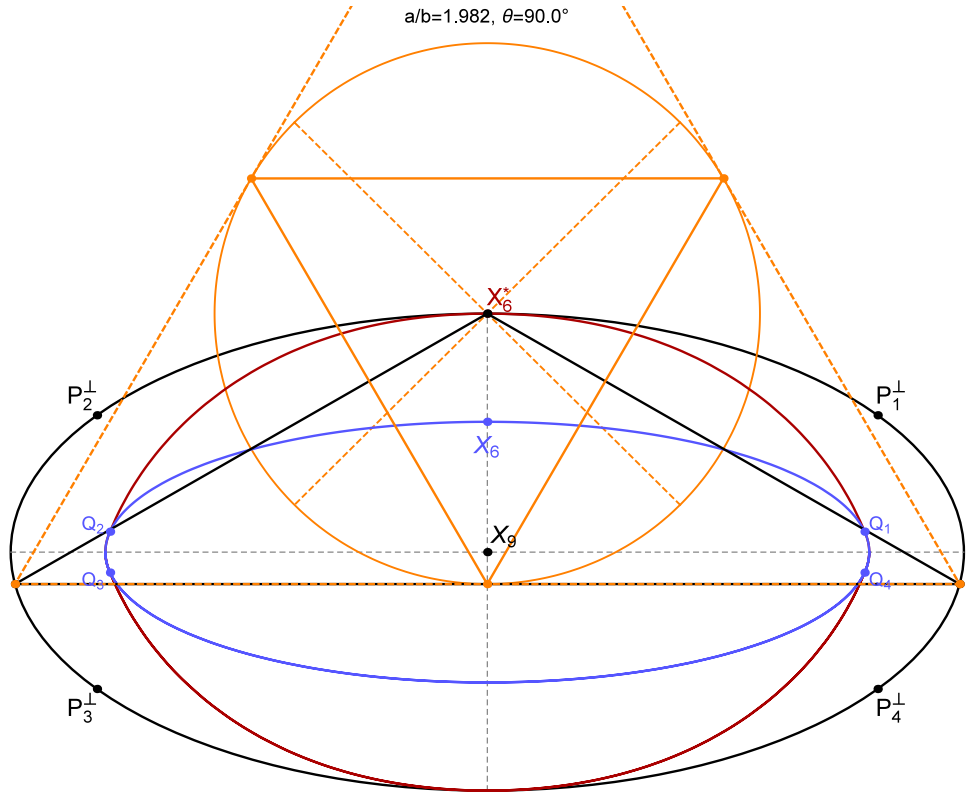}
    \caption{At $a/b=\alpha_{eq}{\simeq}1.982$, the Orthic is an equilateral triangle when a 3-periodic vertex lies on a top or bottom vertex of the EB. Therefore its CB is a circle.}
    \label{fig:cb_ort_equi}
\end{figure}

\subsection{Summary}

Table~\ref{tab:cb_summary} summarizes the CBs discussed above, their centers, and their loci.

\begin{table}[H]
\begin{tabular}{|c|c|c|}
\hline
Triangle & Center & Elliptic Locus \\
\hline
3-Periodic & $X_9$ & n/a \\
Excentral & $X_{168}$ & No\\
ACT & $X_{7}$ & Yes \\
Medial & $X_{142}$ & Yes \\
Orthic & $X_6^*$ & No \\
\hline
\end{tabular}
\caption{CBs mentioned in this Section, their Centers and loci types.}
\label{tab:cb_summary}
\end{table}

\subsection{Circumbilliard of the Poristic Family}

The Poristic Triangle Family is a set of triangles (blue) with fixed Incircle and Circumcircle \cite{gallatly1914-geometry}. It is a cousing of the 3-periodic family in that by definition its Inradius-to-Circumradius $r/R$ ratio is constant.

Weaver \cite{weaver1927-poristic} proved the Antiorthic Axis\footnote{The line passing through the intersections of reference and Excentral sidelines \cite{mw}.} of this family is stationary. Odehnal showed the locus of the Excenters is a circle centered on $X_{40}$ and of radius $2R$ \cite{odehnal2011-poristic}. He also showed that over the family, the locus of the Mittenpunkt $X_9$ is a circle whose radius is $2{d^2}(4R+r)$ and center is $X_1 + (X_1 - X_3) (2 R - r)/(4 R + r)$, where $d=|X_1X_3|=\sqrt{R(R-2{r})}$ \cite[page 17]{odehnal2011-poristic}.

Let $\rho=r/R$ and $a_9,b_9$ be the semi-axis lengths of the Circumbilliard a poristic triangle. As shown in Figure~\ref{fig:cb-poristic}:

\begin{theorem}
The ratio $a_9/b_9$ is invariant over the family and is given by:

\begin{equation*}
\frac{a_9}{b_9}=\sqrt{\frac{\rho^2+2 (\rho+1)\sqrt{1-2\rho} +2}{\rho (\rho+4)}}
\end{equation*}
\noindent where $\rho=r/R$.
\label{thm:poristic}
\end{theorem}

\begin{proof}
The following expression for $r/R$ has been derived for the 3-periodic family of an $a,b$ EB \cite[Equation 7]{garcia2020-new-properties}:

\begin{equation}
  \rho = \frac{r}{R} =  \frac{2(\delta-b^2)(a^2-\delta)}{c^4}
\end{equation}
Solving the above for $a/b$ yields the result.
\end{proof}

\begin{figure}
    \centering
    \includegraphics[width=\textwidth]{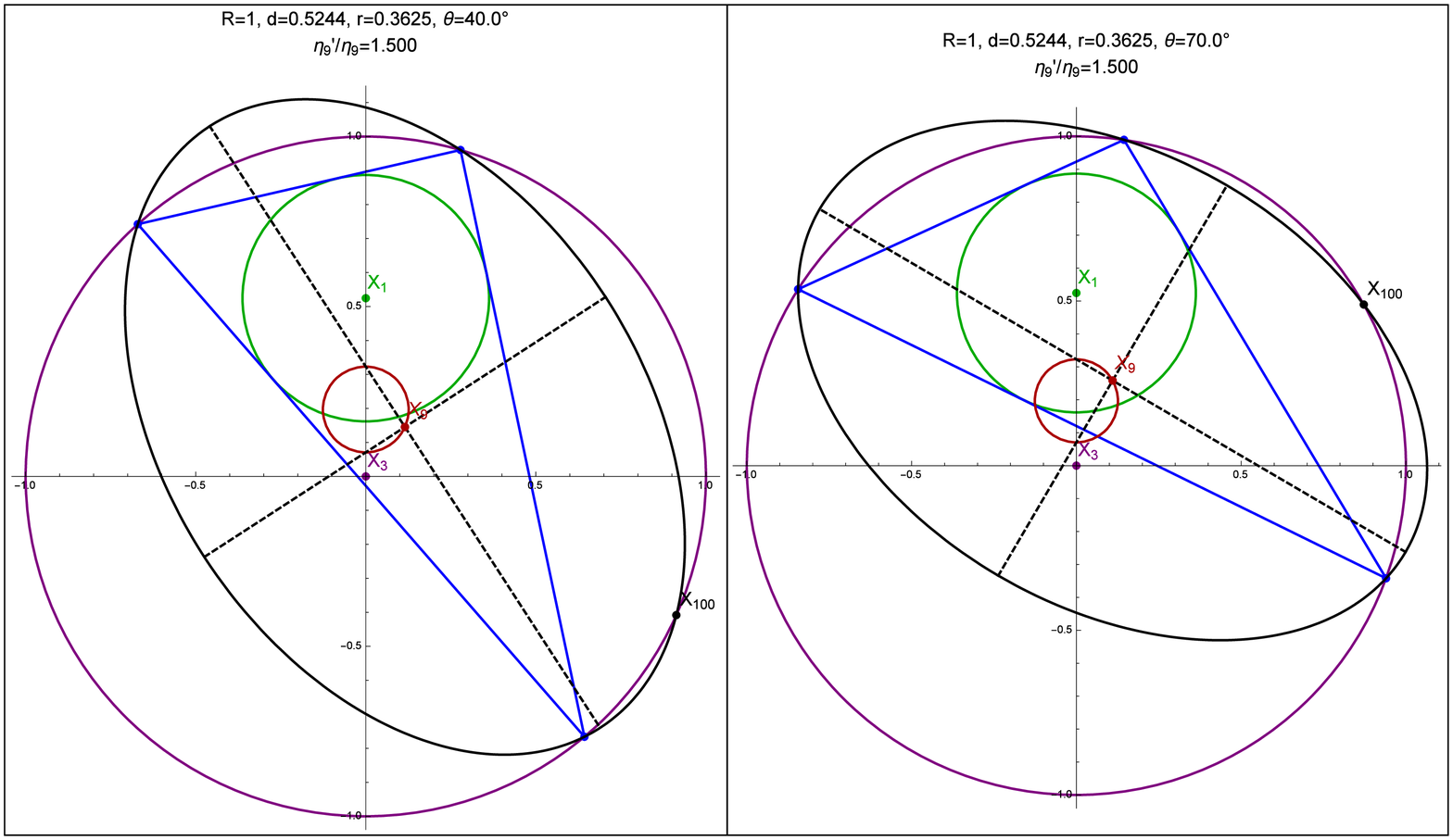}
    \caption{Two configurations (left and right) of the Poristic Triangle Family (blue), whose Incircle (green) and Circumcircle (purple) are fixed. Here $R=1,r=0.3625$. Over the family, the Circumbilliard (black) has invariant aspect ratio, in this case $a_9/b_9{\simeq}1.5$. Also shown is the circular locus of $X_9$ \cite[page 17]{odehnal2011-poristic}. \textbf{Video}: \cite[PL\#07]{reznik2020-playlist-circum}.}
    \label{fig:cb-poristic}
\end{figure}

\section{Conclusion}
\label{sec:conclusion}
Videos mentioned above have been placed on a \href{https://bit.ly/379mk1I}{playlist} \cite{reznik2020-playlist-circum}. Table~\ref{tab:playlist} contains quick-reference links to all videos mentioned, with column ``PL\#'' providing video number within the playlist.

\begin{table}
\begin{tabular}{c|l|l}
\href{https://bit.ly/2NOOIOX}{PL\#} & Title & Section\\
\hline

\href{https://youtu.be/tMrBqfRBYik}{01} &
{Mittenpunkt stationary at EB center} & \ref{sec:intro} \\

\href{https://youtu.be/vSCnorIJ2X8}{02} &
{Circumbilliards (CB) of Various Triangles} &
\ref{sec:cb} \\

\href{https://youtu.be/Og7xLgkrLqw}{03} &
{CBs of Derived Triangles and Loci of Centers} &
\ref{sec:cb_derived} \\

\href{https://youtu.be/xyHUwpvAj3g}{04} & CBs of ACT and Medial (separate) &
\ref{sec:cb_derived} \\

\href{https://youtu.be/e-mToZlkHtc}{05} &
CBs of ACT and Medial (superposed) & \ref{sec:cb_derived} \\

\href{https://youtu.be/5KL8st2vIb0}{06} &
{CB of Orthic and Locus of its Mittenpunkt} &
\ref{sec:cb_derived} \\

\href{https://youtu.be/yEu2aPiJwQo}{07} & \makecell[lt]{Invariant Aspect Ratio of \\Circumbilliard of Poristic Family} &
\ref{sec:cb_derived} \\

\end{tabular}
\caption{Videos mentioned in the paper. Column ``PL\#'' indicates the entry within the playlist \cite{reznik2020-playlist-circum}}
\label{tab:playlist}
\end{table}

\section*{Acknowledgments}
We would like to thank Peter Moses and Clark Kimberling, for their prompt help with dozens of questions. We would like to thank Boris Odehnal for his help with some proofs. A warm thanks goes out to Profs. Jair Koiller and Daniel Jaud who provided critical editorial help.

The second author is fellow of CNPq and coordinator of Project PRONEX/ CNPq/ FAPEG 2017 10 26 7000 508.

\bibliographystyle{spmpsci}
\bibliography{elliptic_billiards_v3,authors_rgk_v1} 

\appendix
\section{Computing a  Circumconic}
\label{app:circum-linear}
Let a Circumconic have center $M=(x_m,y_m)$ Equation~\eqref{eqn:e0} is subject to the following 5 constraints\footnote{If $M$ is set to $X_9$ one obtains the Circumbilliard.}: it must be satisfied for vertices $P_1,P_2,P_3$, and its gradient must vanish at $M$:

\begin{align*}
f(P_i)=&\;0,\;\;\;i=1,2,3\\
\frac{dg}{dx}(x_m,y_m)=&\;c_1+c_3 y_m+2c_4 x_m=0\\
\frac{dg}{dy}(x_m,y_m)=\;&c_2+c_3 x_m+2c_5 y_m=0
\end{align*}

Written as a linear system:

$$
\left[
\begin{array}{ccccc}
x_1&y_1&x_1 y_1&x_1^2&y_1^2\\
x_2&y_2&x_2 y_2&x_2^2&y_2^2\\
x_3&y_3&x_3 y_3&x_3^2&y_3^2\\
1&0&y_m&2\,x_m&0\\
0&1&x_m&0&2\,y_m
\end{array}
\right] .
\left[\begin{array}{c}c_1\\c_2\\c_3\\c_4\\c_5\end{array}\right] =
\left[\begin{array}{c}-1\\-1\\-1\\0\\0\end{array}\right]
$$

Given sidelenghts $s_1,s_2,s_3$, the coordinates of $X_9=(x_m,y_m)$ can be obtained by converting its Trilinears $\left(s_2 + s_3 - s_1 :: ...\right)$ to Cartesians \cite{etc}. 

Principal axes' directions are given by the eigenvectors of the Hessian matrix $H$ (the jacobian of the gradient), whose entries only depend on $c_3$, $c_4$, and $c_5$:

\begin{equation}
H = J(\nabla{g})=\left[\begin{array}{cc}2\,c_4&c_3\\c_3&2\,c_5\end{array}\right]
\label{eqn:hessian}
\end{equation}

The ratio of semiaxes' lengths is given by the square root of the ratio of $H$'s eigenvalues:

\begin{equation}
a/b=\sqrt{\lambda_2/\lambda_1}
\label{eqn:ratiolambda}
\end{equation}

Let $U=(x_u,y_u)$ be an eigenvector of $H$. The length of the semiaxis along $u$ is given by the distance $t$ which satisfies:

$$
g(M + t\,U) = 0
$$

This yields a two-parameter quadratic $d_0 + d_2 t^2$, where:

$$
\begin{array}{cll}
d_0 & = & 1 + c_1 x_m + c_4 x_m^2 + c_2 y_m + c_3 x_m y_m + c_5 y_m^2 \\ 
d_2 & = & c_4 x_u^2 + c_3 x_u y_u + c_5 y_u^2
\end{array}
$$

The length of the semi-axis associated with $U$ is then $t=\sqrt{-d_0/d_2}$. The other axis can be computed via \eqref{eqn:ratiolambda}.

The eigenvectors (axes of the conic) of $H$ are given by the  zeros of the quadratic form
\begin{align*}
   q(x,y)= c_3(y^2-x^2)+2(c_2-c_5)xy
\end{align*}



\section{Table of Symbols}
\label{app:symbols}
Tables~\ref{tab:kimberling} and \ref{tab:symbols} lists most Triangle Centers and symbols mentioned in the paper.

\begin{table}[H]
\scriptsize
\begin{tabular}{|c|l|l|}
\hline
Center & Meaning & Note\\
\hline
$X_1$ & Incenter & Locus is Ellipse \\
$X_2$ & Barycenter & Perspector of Steiner Circum/Inellipses \\
$X_3$ & Circumcenter & Locus is Ellipse, Perspector of $M$ \\
$X_4$ & Orthocenter & \makecell[tl]{Exterior to EB when 3-periodic is obtuse} \\
$X_5$ & Center of the 9-Point Circle & \\
$X_6$ & Symmedian Point & Locus is Quartic \cite{garcia2020-ellipses} \\
$X_6^*$ & $X_9$ of Orthic & Detached from $X_6$ locus for obtuse triangles \\
$X_7$ & Gergonne Point & Perspector of Incircle \\
$X_8$ & Nagel Point & Perspector of $I_9$, $X_1$ of ACT Incircle \\
$X_9$ & Mittenpunkt & Center of (Circum)billiard \\
$X_{10}$ & Spieker Point & Incenter of Medial \\
$X_{11}$ & Feuerbach Point & on confocal Caustic \\
$X_{40}$ & Bevan Point & $X_3$ of Excentral \\
$X_{69}$ & $X_6$ of the ACT & Perspector of $I_3$ \\
$X_{100}$ & Anticomplement of $X_{11}$ & On Circumcircle and EB, $J_{exc}$ center \\
$X_{142}$ & $X_9$ of Medial & Midpoint of $X_9{X_{7}}$, lies on $L(2,7)$ \\ 
$X_{144}$ & Anticomplement of $X_7$ &\makecell[tl]{Perspector of ACT and its Intouch Triangle} \\
$X_{168}$ & $X_9$ of the Excentral Triangle &  Non-elliptic Locus\\
\hline
$L(2,7)$ & ACT-Medial Mittenpunkt Axis & Line $\mathcal{L}_{663}$  \cite{etc_central_lines} \\
\hline
\end{tabular}
\caption{Kimberling Centers and Central Lines mentioned in the paper.}
\label{tab:kimberling}
\end{table}

\begin{table}
\scriptsize
\begin{tabular}{|c|l|l|}
\hline
Symbol & Meaning & Note\\
\hline
$a,b$ & EB semi-axes & $a>b>0$\\ 
$P_i,s_i$ & Vertices and sidelengths of 3-periodic & invariant $\sum{s_i}$ \\
$P_i'$ & Vertices of the Excentral Triangle & \\
$a_9,b_9$ & Semi-axes of Poristic Circumbilliard & \\
$r,R,\rho$ & Inradius, Circumradius, $r/R$ & $\rho$ is invariant \\
$\delta$ & Oft-used constant & $\sqrt{a^4-a^2 b^2+b^4}$ \\
$d$ & Distance $|X_1{X_3}|$ & $\sqrt{R(R-2r)}$\\
$\alpha$ & EB aspect ratio & $a/b$ \\
$\alpha_4$ & \makecell[tl]{$a/b$ threshold for obtuse 3-Periodics} & $ \sqrt{2\,\sqrt {2}-1}$\\
$\alpha_{eq}$ & $a/b$ for equilateral Orthic & $\sqrt{4\sqrt{3}-3}$\\
$P^\perp$ & Obtuse 3-periodic limits on EB & \\
$x^*,y^*$ & where $X_6^*$ detaches from $X_6$ locus & Occurs when some $P_i$ is at $P^\perp$\\
\hline
\end{tabular}
\caption{Symbols mentioned in the paper.}
\label{tab:symbols}
\end{table}

\end{document}